\documentclass[12pt,reqno]{amsart}
\usepackage {amssymb,amsfonts}
\usepackage {amsmath}
\usepackage{amsthm}
\usepackage{mathtools}
\usepackage[old]{old-arrows} 
\usepackage{stmaryrd}

\usepackage[utf8]{inputenc}
\usepackage[english]{babel}

\usepackage{tikz}
\usepackage{tikz-cd}
\usetikzlibrary{babel}

\usepackage{graphicx}
\usepackage[alphabetic]{amsrefs}
\usepackage[bookmarks=false,hypertexnames=false,colorlinks, citecolor=blue]{hyperref}
\usepackage{bookmark}
\usepackage{verbatim,color,geometry}
\usepackage[colorinlistoftodos, textsize=small]{todonotes}
\makeatletter 
\@mparswitchfalse%
\makeatother
\reversemarginpar

\usepackage[shortlabels]{enumitem}

\newtheorem{prop}[equation]{Proposition}
\newtheorem{thm}[equation]{Theorem}
\newtheorem{defn-thm}[equation]{Theorem-Definition}
\newtheorem{defn-lem}[equation]{Lemma-Definition}
\newtheorem{cor}[equation]{Corollary}

\newtheorem*{thm*}{Theorem}
\newtheorem*{cor*}{Corollary}

\theoremstyle{definition}

\newtheorem{example}[equation]{Example}
\newtheorem{rem}[equation]{Remark}

\newtheorem*{claim*}{Claim}
\newtheorem*{rem*}{Remark}

\newcommand{\su}{\subset}

\newcommand{\ttto}{\dashrightarrow} 
\newcommand{\tto}{\longrightarrow}

\newcommand{\mapstto}{\longmapsto}

\newcommand{\wt}{\widetilde}

\newcommand{\ov}{\overline}

\newcommand{\lc}{\lceil}
\newcommand{\rc}{\rceil}

\newcommand{\PP}{\mathbb{P}}

\newcommand{\OO}{\mathcal{O}}

\newcommand{\cc}{\mathfrak{c}}

\newcommand{\pp}{\mathfrak{p}}
\newcommand{\fq}{\mathfrak{q}}

\newcommand{\de}{\delta}
\newcommand{\De}{\Delta}

\newcommand{\ka}{\kappa}

\DeclareMathOperator{\Frac}{Frac}

\DeclareMathOperator{\Spec}{Spec}

\numberwithin{equation}{section}

\begin{document}

\title{On regular but non-smooth integral curves}

\author{Cesar Hilario}
\address{Mathematisches Institut, Heinrich-Heine-Universit\"at, 40204 D\"usseldorf, Germany}
\email{cesar.hilario.pm@gmail.com}

\author{Karl-Otto St\"ohr}
\address{IMPA, Estrada Dona Castorina 110, 22460-320 Rio de Janeiro, Brazil}
\email{stohr@impa.br}


\dedicatory{July 19, 2024}

\begin{abstract}
Let $C$ be a regular geometrically integral curve over an imperfect field $K$ and assume that it admits a non-smooth point $\mathfrak{p}$ which --- seen as a prime of the separable function field $K(C)|K$ --- is non-decomposed in the base field extension $\overline{K} \otimes_K K(C)|\overline{K}$.
In this paper we establish a bound for the number of iterated Frobenius pullbacks needed in order to transform $\mathfrak{p}$ into a rational point.
This provides an algorithm to compute geometric $\delta$-invariants of non-smooth points and a procedure to construct fibrations with moving singularities of  prescribed $\delta$-invariants.
We show that the bound is sharp in characteristic 2.
We further study the geometry of a pencil of plane projective rational quartics in characteristic 2 whose generic fibre attains our bound.
On our way, we prove several results on separable and non-decomposed points that might be of independent interest.
\end{abstract}

\maketitle

\setcounter{tocdepth}{1}

\section{Introduction}\label{2024_07_01_23:55}
Bertini's theorem on variable singular points, also known as the Bertini--Sard theorem, is nowadays one of the most used theorems in algebraic geometry.
In its modern version, it states that
in characteristic zero almost every fibre of a dominant morphism $\phi:T\to B$ of integral smooth algebraic varieties over an algebraically closed field $k$ is smooth.
This is no longer the case in positive characteristic, as already noted by Zariski \cite{Zar44} in the 1940s. 
The most familiar counterexamples are the quasi-elliptic fibrations that arise in the classification of algebraic surfaces by Bombieri and Mumford in characteristics $2$ and $3$ (see \cites{BM76,Lan79}).

From the point of view of Grothendieck's scheme theory, the generic fibre 
\[ C := T \times_B \Spec k(B) \] 
of the fibration $\phi: T \to B$ is a \emph{regular} scheme over the function field $K := k(B)$ of the base $B$, yet it may happen that the geometric generic fibre 
\[ C \otimes_K \overline K = C \times_{\Spec(K)} \Spec \ov K \]
is \emph{not regular}.
Such non-regularity occurs precisely when every special fibre is singular,
and so this reveals a deep connection between the failure of Bertini's theorem and the existence of regular schemes $C$ defined over an imperfect field $K$ that are \emph{non-smooth}, i.e., for which the base extension $C\otimes_K \ov K$ becomes non-regular.
Such existence represents a striking feature of geometry in positive
characteristic, 
that results from the fact that over imperfect fields the notions of smoothness and regularity differ: every smooth variety (i.e., smooth scheme of finite type over a field) is regular, but not every regular variety is smooth.\footnote{However, a rational point is smooth if and only if it is regular \cite[Corollary~2.6]{FaSc20}. For a discussion of regularity versus smoothness we refer to \cite[Chap.~11.28]{Mat80}.}
In several areas such as birational geometry, and particularly in the Minimal Model Program,
these regular but non-smooth schemes cause difficulties when one tries to apply zero characteristic techniques to positive characteristic situations;
an explicit example: del Pezzo fibrations, where the picture seems more involved in characteristic $2$ (see \cite[p.\,404]{Ma16}, \cite[Remark~1.2]{Ko91}, and \cite{FaSc20}).
As a result, much effort has been devoted to understand this behaviour and to bound its occurrence (see e.g.\,\cites{Sc08, JW21, PW22,HiSc23}).

In this paper we explore the above phenomenon in the specific
situation where the variety is a regular geometrically integral curve $C$ over an imperfect field $K$.
Note that since $C$ has dimension one, regularity is the same as normality.
If $C|K$ is the generic fibre of a fibration $f:T\to B$ then its closed points correspond bijectively to the horizontal divisors on the total space $T$; a closed point is non-smooth if and only if the corresponding divisor is a moving singularity of the fibration \cite[Section~1]{SiSt16}.

Our approach to the non-smoothnes of $C$ relies on a central tool in geometry in positive characteristic: Frobenius pullbacks.
As a non-smooth point $\pp$ of $C$ cannot be \emph{smoothed} by performing Frobenius pullbacks, because its images in the sequence of iterated Frobenius pullbacks
\[ C\to C^{(p)} \to C^{(p^2)} \to C^{(p^3)} \to \cdots \]
are non-regular \cite[Lemma~2.2]{Sal11} and therefore non-smooth, 
we consider instead the images $\pp_n \in C_n$ of $\pp$ in the sequence of regular integral curves 
\[ C_0 = C \to C_1 \to C_2 \to C_3 \to \cdots \]
obtained by passing to the normalizations $C_n$ of the Frobenius pullbacks $C^{(p^n)}$.
Our main result, stated below as Theorem~\ref{2023_05_31_19:35}, provides an explicit description of an integer $n$ for which the image point $\pp_n$ is separable (i.e., the residue field extension $\ka(\pp_n)|K$ is separable) and a fortiori smooth.
In particular, if $\pp$ is non-decomposed in the base extension $C \otimes_K \overline K$, that is, if there is only one point of $C \otimes_K \overline K$ lying over $\pp$, then the separable point $\pp_n$ is actually rational (see Corollary~\ref{2023_20_02_16:05}).

\begin{thm}[see Theorem~\ref{2022_03_04_12:45}]\label{2023_05_31_19:35}
    Let $C$ be a regular geometrically integral curve over a field $K$ of characteristic $p>0$.
    Let $\pp$ be a non-smooth point of geometric $\de$-invariant $\de(\pp)>0$.
    \begin{enumerate}[\upshape (i)]
        \item The image $\pp_n$ of $\pp$ in the normalization $C_n$ of the $n$th Frobenius pullback $C^{(p^n)}$ of $C$ is separable for all $n \geq \log_p \big( 2 \, \frac{\de(\pp)}{[\ka(\pp):K]_{sep}} + 1 \big)$; moreover, if the integer $\frac 2{p-1} \frac{\de(\pp)}{[\ka(\pp):K]_{sep}}$ is not a sum of consecutive $p$-powers then $\pp_n$ is separable for all 
        \mbox{$n \geq \log_p \big( 2 \, \frac{\de(\pp)}{[\ka(\pp):K]_{sep}} + 1 \big) - 1$}.
        \item \label{2023_05_31_21:15}
        Assume in addition that $\pp$ is non-decomposed in $C \otimes_K \overline K$. 
        Then the image $\pp_n$ is rational for all \mbox{$n \geq \log_p \big( 2  {\de(\pp)} + 1 \big)$}; 
        moreover, if the integer $\frac 2{p-1} {\de(\pp)}$ is not a sum of consecutive $p$-powers then $\pp_n$ is rational for all 
        $n \geq \log_p \big( 2{\de(\pp)} + 1 \big) - 1$.
    \end{enumerate}
\end{thm}

Here $[\ka(\pp):K]_{sep}$ denotes the separable degree of the residue field extension $\ka(\pp)|K$.
Note that $\frac 2{p-1} \frac{\de(\pp)}{[\ka(\pp):K]_{sep}}$ and $\frac 2{p-1} \de(\pp)$ are indeed integers (see Corollary~\ref{2023_02_23_14:20}).

Our motivation originates from the following observation:
if an integer $n$ is known such that the image $\pp_n$ is rational, then an algorithm developed in \cite{BedSt87} by Bedoya--St\"ohr can be applied to compute the geometric 
$\de$-invariant $\de(\pp)$ of $\pp$ and several other invariants associated to $\pp$. 

To prove our results we employ methods from the theory of algebraic function fields (see \cite[II.7.4]{EGA}.
Let $F|K = K(C)|K$ be the function field of the regular integral curve $C|K$. The function fields of the iterated Frobenius pullbacks $C^{(p^n)}|K$ and of their normalizations $C_n|K$ are the iterated Frobenius pullbacks of $F|K$: 
\[ F_n|K = KF^{p^n}|K, \quad  (n=0,1,2,\dots). \]
In order to study the sequence of normalized curves $C_n$ we study the descending chain of function fields 
\[ F = F_0 \supset F_1 \supset F_2 \supset F_3 \supset \cdots. \]
A closed point of the curve $C$ and its image in $C_n$ correspond to a prime $\pp$ of $F|K$ and its restriction $\pp_n$ to $F_n|K$.

As a main application of the theorem we get a procedure to construct regular integral curves $C|K$ equipped with non-decomposed non-smooth closed points $\pp$ of a given geometric $\de$-invariant, 
or equivalently,  
a procedure to construct for each natural number $d$ the function fields $F|K$ equiped with non-decomposed singular primes $\pp$ such that $\delta(\pp)=d$.
To this end let $n=\lc \log_p(2d+1) \rc$ or $n = \lc \log_p (2d+1) \rc - 1$ be the corresponding bound in the theorem.
Each such pair $(F|K,\pp)$ can be obtained by starting with a function field $F_n|K$ equipped with a rational prime $\pp_n$, and by constructing an ascending length-$n$ chain of purely inseparable extensions of function fields
\[ F_n \su F_{n-1} \su F_{n-2} \su \cdots \su F_0 = F \]
equipped with the (uniquely determined) primes $\pp_n,\pp_{n-1},\pp_{n-2},\dots, \pp_0 = \pp$ lying over the rational prime $\pp_n$, such that $F_i|K$ is the Frobenius pullback of $F_{i-1}|K$ for each \mbox{$i=n,n-1,\dots,1$}.
The generators of the purely inseparable extensions $F_{i-1}|F_i$ of degree $p$ are obtained by applying the Riemann-Roch theorem.
With the Bedoya--St\"ohr algorithm in mind the generators have to be chosen carefully, so that the sequence of geometric $\de$-invariants $d_n = 0 \leq d_{n-1} \leq \dots \leq d_0$ ends with $d=d_0$.
Looking for decomposed non-smooth points we have to start our procedure with separable but non-rational primes. 
Our method is illustrated in
\cite{HiSt23} 
and \cite{HiSt24},
where curves of arithmetic genus $g=3$ equipped with non-smooth points of geometric $\de$-invariant $d=3$ are constructed.

We show that the bound in the theorem is sharp in characteristic $p=2$.
Furthermore,
on our way we obtain several results on separable and non-decomposed closed points
that might be of independent interest
(see Proposition~\ref{2023_02_06_13:00} and Remark~\ref{2023_05_31_20:55}).

In the last section of the paper we study  a pencil of singular rational quartics in characteristic~$p=2$,
whose generic fibre $C|K$ attains the sharp bound 
for $\de(\pp)=3$.
We discuss the geometry of the fibration in detail, and we further find its minimal regular model, which by a theorem of Lichtenbaum and Shafarevich is uniquely determined by the function field of $C|K$.

\subsection*{Acknowledgements} We thank Stefan Schr\"oer for several suggestions and feedback on an earlier draft of this manuscript.
This project was started while the first author was a PhD student at IMPA, funded by CAPES Brazil, and was completed while he was a member of the research training group \textit{GRK 2240: Algebro-Geometric Methods in Algebra, Arithmetic and Geometry}, which is funded by the Deutsche Forschungsgemeinschaft.
Parts of it were also written during a research visit he made to IMPA in March 2023.

\section{Non-smooth points of regular integral curves} \label{2022_04_24_16:10}
There is a one-to-one correspondence between the regular proper integral curves over (the spectrum of) a given field $K$ and the one-dimensional function fields with base field $K$ (see \cite[II.7.4]{EGA}).
The function field $F|K$ corresponding to such a curve $C|K$ is the local ring at the generic point. 
Conversely, the points of the curve $C$ different from the generic point, i.e., the closed points of $C$, are the primes $\pp$ of the function field $F|K$, and their local rings $\OO_{\pp}$ are just the (discrete) valuation rings of $F|K$. If $U$ is a non-empty open subset of $C$, that is, the complement of a finite set of closed points, then the space $\Gamma(U,\OO_C)$ of local sections of the structure sheaf $\OO_C$ is the intersection of the local rings $\OO_{\pp}$ of the closed points $\pp \in U$.

In this section we assume that $F|K$ is a one-dimensional separable function field 
in positive characteristic $p$. 
This means that $F|K$ is a separable finitely generated field extension of transcendence degree $1$, such that $K$ is algebraically closed in $F$.
The latter assumption together with the separability of $F|K$ mean that the corresponding regular proper integral curve $C|K$ is geometrically integral, i.e., it remains integral under algebraic extensions of the base field. 

Let $\pp$ be a prime of $F|K$ and consider its (discrete) valuation ring $\OO_\pp$.
If $K'$ is an algebraic extension of the base field $K$, then the tensor product $K' {\cdot} \OO_\pp := K'\otimes_K \OO_\pp$ is a semilocal domain with fraction field $\Frac(K' {\cdot} \OO_\pp) =  K' {\cdot} F:= K'\otimes_K F$ that coincides with the finite intersection of its localizations (i.e., the localizations at its maximal ideals).
Base extensions of local rings are therefore semilocal.
For this reason, in order to study the behaviour of  $\pp$ under base field extensions it is convenient to work with the semilocal domains of $F|K$ rather than with its local rings.

Let $\OO$ be a semilocal domain of $F|K$ (where $K \subset \Frac(\OO) = F$), which is equal to the finite intersection of its localizations, say
\[ \OO = \OO_1\cap \dots \cap \OO_r \, . \]
Let $K'$ be an algebraic extension of the base field $K$.
The \emph{$K'$-singularity degree} of $\OO$, which is defined as the $K'$-codimension of the extended semilocal ring $K'{\cdot} \OO = K' \otimes _K \OO$ in its integral closure $\wt{K'{\cdot} \OO}$, is finite (see \cite[Theorem~1]{Ros52}) and equal to  the sum of the $K'$-singularity degrees of the localizations, i.e.,
\begin{equation}\label{2023_02_22_20:25}
   \dim_{K'} \wt{K' {\cdot} \OO}/K' {\cdot}\OO = \sum_{i=1}^r \dim_{K'} \wt{K'{\cdot}\OO_i}/K'{\cdot}\OO_i
\end{equation}
(see \cite[p.\,172]{Ros52}); indeed, the canonical homomorphism $\wt{K' {\cdot}\OO} \to \bigoplus_{i=1}^r \wt{K'{\cdot}\OO_i} / K'{\cdot}\OO_i$ has kernel $K'{\cdot}\OO$ and is surjective, as follows by applying the Chinese remainder theorem to the conductor ideals of the rings $K'{\cdot}\OO_i$.

If $K''$ is an algebraic extension field of $K'$ then the $K''$-singularity degree of $\OO$ is equal to the sum of the $K'$-singularity degree of $\OO$ and the $K''$-singularity degree of $\wt{K'{\cdot}\OO}$, as can be seen from
\begin{align}
    \dim_{K'} \wt{K'{\cdot}\OO}/K'{\cdot}\OO &= \dim_{K''} (K''\otimes_{K'} \wt{K'{\cdot} \OO}) / (K'' \otimes_{K'} K'{\cdot}\OO) \notag \\
    &= \dim_{K''} K'' {\cdot} \wt{K'{\cdot}\OO} / K'' {\cdot}\OO \notag \\
    &= \dim_{K''} \wt{K''{\cdot}\OO} / K''{\cdot}\OO - \dim_{K''} (\wt{K''{\cdot}\OO} / K''{\cdot} \wt{K'{\cdot}\OO}). \label{2023_02_22_20:30}
\end{align}
The \textit{geometric singularity degree}%
\footnote{Another name in the literature is ``geometric $\delta$-invariant''.}
of a prime $\pp$, denoted $\de(\pp)$, is defined as the $\ov K$-singularity degree of its local ring $\OO_\pp$, where $\ov K$ is the algebraic closure of $K$. 
The prime $\pp$ is called \textit{singular} if $\de(\pp)>0$, i.e., $\ov K {\cdot} \OO_\pp \subsetneqq \wt{\ov K {\cdot} \OO_\pp}$, i.e., 
$\pp$ is a non-smooth point of the corresponding regular integral curve $C|K$.%
\footnote{It might be tempting to use the term ``non-smooth prime'', but the term ``singular prime'' is already in use in the literature on function fields.}

By Rosenlicht's genus drop formula (see \cite[Theorem~11]{Ros52}) the genus of the extended function field $\ov K F|\ov K$ is equal to 
\begin{equation}\label{2022_03_04_18:10}
    \ov g = g - \sum_\pp \de(\pp)
\end{equation}
where $g$ is the genus of the function field $F|K$ and the sum is taken over the singular primes $\pp$ of $F|K$.
The genus drop $g - \ov g$ is divisible by $(p-1)/2$ if $p>2$, by a theorem of Tate \cite{Tate52} (see also \cite{Sc09} for a modern treatment).
It follows that the function field $F|K$ is \textit{conservative} (i.e., $\ov g = g$) if and only if it does not admit singular primes,
or equivalently, if and only if the corresponding regular integral curve $C|K$ is smooth.

For every non-negative integer $n$ we consider the $n$th \emph{Frobenius pullback}
$F_n|K := F^{p^n}{\cdot} K|K$.
This function field is uniquely determined by the property that the extension $F|F_n$ is purely inseparable of degree $p^n$ (see \cite[p.\,33]{Sti78}). 
Let $\pp_n$ be the restriction of the prime $\pp$ to $F_n$, and let $\pp^{(n)}$ be the only extension of $\pp$ to the purely inseparable base field extension $F^{(n)} := K^{p^{-n}} \! {\cdot} F$.
The valuation ring $\OO_{\pp^{(n)}}$ of $\pp^{(n)}$ is the integral closure $\wt{K^{p^{-n}} \! {\cdot} \OO_\pp}$ of the domain $K^{p^{-n}} \! {\cdot} \OO_\pp$ in its field of fractions $K^{p^{-n}} \! {\cdot} F$.
The $n$th Frobenius map $z \mapsto z^{p^n}$ defines an isomorphism between the function fields $F^{(n)}| K^{p^{-n}}$ and $F_n|K$ which maps $\pp^{(n)}$ to $\pp_n$.
As the ramification index of $\pp^{(n)}|\pp_n$ is equal to $p^n$ we get $e(\pp^{(n)}|\pp) \cdot e (\pp|\pp_n) = p^n$ and therefore
\[ e(\pp^{(n+1)}|\pp^{(n)}) \cdot e(\pp_n | \pp_{n+1}) = p \; \text{ for each $n$.} \]
As the field extensions $F_n|F_{n+1}$ are purely inseparable of degree $p$, each residue field extension $\ka(\pp_{n})|\ka(\pp_{n+1})$ is purely inseparable of degree $[\ka(\pp_n):\ka(\pp_{n+1})] \in \{1,p\}$.

In this section we ask for an integer $n$ such that the restricted prime $\pp_n$ is rational, i.e., such that its degree $\deg(\pp_n)=[\ka(\pp_n):K]$ is equal to $1$.
If such an integer is known, then the algorithm developed in \cite{BedSt87} can be applied to compute several local invariants of $F|K$
such as the singularity degrees of $\pp$ or the orders of differentials at $\pp$.
For each non-negative integer $n$ we denote by
\[ \De_n = \De(\pp_n) := \dim_{K^{1/p}} (\wt{K^{1/p} {\cdot} \OO_{\pp_n}} / K^{1/p} {\cdot} \OO_{\pp_n} ) \]
the $K^{1/p}$-singularity degree of $\pp_n$.

\begin{prop}\label{2024_07_14_19:10}
With the above notation, the non-negative integers $\De_0,\De_1,\De_2,\dots$ are divisible by $(p-1)/2$ if $p>2$. 
Moreover 
\[ \De_{n+1} \leq p^{-1} \De_n \; \text{ for each $n$}. \]
\end{prop}

\begin{proof}
See \cite[Corollary~2.4 and Proposition~3.5]{St88}.
\end{proof}
In particular $\De_n = 0$ for $n$ sufficiently large, or more precisely
\begin{equation*}
	\De_n = 0  \; \text{ whenever } \; p^n > 
\begin{cases}
    \De_0 & \text{if $p=2$ or $3$}, \\
    \frac{2}{p-1} \De_0 & \text{if $p>2$.}
\end{cases}
\end{equation*}

\begin{prop}\label{2022_07_17_12:05}
For every prime $\pp$ of $F|K$ the following equality holds
\[ \de(\pp) = \De_0 + \De_1 + \De_2 + \cdots.\]
In particular, the geometric singularity degree $\de(\pp)$ is a multiple of $(p-1)/2$.
\end{prop}

\begin{proof}
    By using the $n$-th Frobenius map we see that $\De_n$ is equal to the $K^{p^{-(n+1)}}$-singularity degree of $\pp^{(n)}$.
    Hence by considering the chain of local rings
    \[ \OO_\pp = \OO_{\pp^{(0)}} \su \OO_{\pp^{(1)}} \su \OO_{\pp^{(2)}} \su \cdots \su \OO_{\pp^{(n)}} \]
    we deduce that the $K^{p^{-(n+1)}}$-singularity degree of $\pp$ is equal to the sum $\De_0 + \De_1 + \dots + \De_n$.
    As $K^{p^{-\infty}} := \bigcup K^{p^{-n}}$ and therefore
    \[ K^{p^{-n}} \! {\cdot} F = \bigcup_{n = 0}^\infty K^{p^{-n}} \! {\cdot} F = \varinjlim K^{p^{-n}} \! {\cdot} F, \] 
    we conclude that the $K^{p^{-\infty}}$-singularity degree of $\pp$ is equal to $\De_0 + \De_1 + \cdots$.
    
    Let $\pp^{(\infty)}$ be the only extension of $\pp$ to the purely inseparable base field extension $K^{p^{-\infty}} \! {\cdot} F$.
    As the algebraic closure $\ov K$ is separable over $K^{p^{-\infty}}$, the prime $\pp^{(\infty)}$ is non-singular and so the geometric singularity degree $\de(\pp)$ of the prime $\pp$ coincides with its $K^{p^{-\infty}}$-singularity degree
    (see also the proof of Proposition~\ref{2023_02_06_13:00} below).
\end{proof}

By applying the proposition to the restricted prime $\pp_n$ for each non-negative integer $n$ we obtain

\begin{cor} \label{2022_07_17_12:55} $ $
For each prime $\pp$ the following assertions hold
    \begin{enumerate}[\upshape (i)]
        \item $\de(\pp_n) = \De_n + \De_{n+1} + \cdots$;
        \item $\De_n = \de(\pp_n) - \de(\pp_{n+1})$;
        \item the prime $\pp_n$ is non-singular if and only if $\De_n = 0$;
        \item if $\De_n = 0$ and $n>0$, then $\de(\pp_{n-1}) = \De(\pp_{n-1})$;
        \item the prime $\pp_n$ is non-singular whenever $p^n>\min\{ \De_0, \frac 2{p-1} \De_0 \}$.
    \end{enumerate}
\end{cor}

Using the genus drop formula~\eqref{2022_03_04_18:10} we get 

\begin{cor}
Let $g_n$ be the genus of the $n$th Frobenius pullback $F_n|K$ of $F|K$. Then
\begin{enumerate}[\upshape (i)]
    \item $g_n - g_{n+1} = \sum_\pp \De_n(\pp)$ for each $n\geq 0$;
    \item $g - \ov g = \sum_{n\geq0} (g_n - g_{n+1})$;
    \item the differences $g_n - g_{n+1}$ are multiples of $(p-1)/2$ and they satisfy $g_{n+1} - g_{n+2} \leq p^{-1} (g_n - g_{n+1})$;
    \item $g_n - g_{n+1} = 0$ if and only if the function field $F_n|K$ is conservative, i.e., $g_n= \ov g$;
    \item the function field $F_n|K$ is conservative whenever $p^n > \min\{ g-g_1, \frac{2}{p-1}(g - g_1) \}$.
\end{enumerate}
\end{cor}

\begin{prop}\label{2021_09_24_16:03}
    Let $\pp$ be a singular prime of 
    $F|K$.
    Then the degree of $\pp_1$ is a divisor of the integer $\frac 2 {p-1} \De(\pp)$.
\end{prop}

\begin{proof}
    Let $K':=K^{1/p}$. As $K'{\cdot} \OO_\pp$ is a Gorenstein ring (see \cite[Theorem~1.1(b)]{St88}) we obtain
    \[ 2\De_0 = \dim_{K'} \wt{K'{\cdot} \OO_\pp}/ \cc'_\pp \]
    where $\cc'_\pp$ denotes the conductor ideal of the domain $K'{\cdot} \OO_\pp$ in its integral closure $\wt{K'{\cdot} \OO_\pp}$.
    As $\wt{K'{\cdot} \OO_\pp}$ is the discrete valuation ring $\OO_{\pp^{(1)}}$, the non-zero ideal $\cc_\pp'$ is a power of the maximal ideal of $\OO_{\pp^{(1)}}$, and so the $K'$-dimension of $\wt{K'{\cdot} \OO_\pp} / \cc_\pp'$ is a multiple of $\deg(\pp^{(1)}) = \deg (\pp_1)$.
    By \cite[Corollary~2.4]{St88} this dimension is even a multiple of $(p-1) \deg (\pp_1)$.
\end{proof}

We say that a prime $\pp$ is \emph{separable} if the residue field extension $\ka(\pp)|K$ is separable. 
Every separable prime is non-singular (see e.g. \cite[Corollary~2.6]{FaSc20}).
The proposition below provides a converse.

\begin{prop}\label{2022_02_06_10:20}
A prime $\pp$ is separable if and only if it is non-singular and $\ka(\pp)=\ka(\pp_1)$.
\end{prop}

\begin{proof}
Since the extension $\ka(\pp)|\ka(\pp_1)$ is purely inseparable, it is trivial if $\pp$ is separable. Thus we may assume $\ka(\pp)=\ka(\pp_1)$. Then \cite{Sti78}, Satz 2 (ii) and Korollar 1 of Satz 4, ensure that $\pp$ is singular if and only if $\ka(\pp)|K$ is inseparable.
\end{proof}

\begin{prop}\label{2022_01_20_15:00}
Let $\pp$ be a prime of $F|K$. Then for sufficiently large $n$ the restricted prime $\pp_n$ is separable and its residue field $\ka(\pp_n)$ is the separable closure of $K$ in $\ka(\pp)$.
\end{prop}

\begin{proof}
As $\ka(\pp) \supseteq \ka(\pp_1) \supseteq \ka(\pp_2) \supseteq \dots \supseteq K$ and $[\ka(\pp):K]$ is finite we deduce that $\ka(\pp_n) = \ka(\pp_{n+1}) = \cdots$ for sufficiently large $n$.
Moreover, for sufficiently large $n$ the prime $\pp_n$ is non-singular by Corollary~\ref{2022_07_17_12:55}, and so $\pp_n$ is separable by Proposition~\ref{2022_02_06_10:20}.
As $\ka(\pp_n)|K$ is separable and $\ka(\pp)|\ka(\pp_n)$ is purely inseparable we conclude that $\ka(\pp_n)$ is the separable closure of $K$ in $\ka(\pp)$.
\end{proof}

We ask for an explicit description of the integers $n$ for which the restricted primes $\pp_n$ are separable. The answer is rather easy if the prime $\pp$ is non-singular.

\begin{prop}\label{2023_01_25_19:00}
Let $\pp$ be a non-singular prime of $F|K$, and let $m:= \log_p [\ka(\pp):K]_{insep}$, i.e., let $p^m$ be the inseparable degree of the residue field extension $\ka(\pp)|K$. Then
\[ [\ka(\pp):\ka(\pp_i)] = p^i \quad \text{for each $i=1,\dots,m$,} \]
and $m$ is the smallest integer such that $\pp_m$ is separable.
\end{prop}

\begin{proof}
In the descending chain
\[ \ka(\pp) \supseteq \ka(\pp_1) \supseteq \ka(\pp_2) \supseteq \dots \supseteq K \]
the extensions $\ka(\pp_n)|\ka(\pp_{n+1})$ are purely inseparable of degree $p$ or $1$ for each $n$.
As the prime $\pp$ and therefore the restricted primes $\pp_n$ are non-singular, it follows from Proposition~\ref{2022_02_06_10:20} that $[\ka(\pp_n):\ka(\pp_{n+1})]=p$ if and only if $\pp_n$ is inseparable.
\end{proof}

An analogous result is much more involved if the prime $\pp$ is singular (see Theorem~\ref{2022_03_04_12:45}).
The reason is that the extension $\ka(\pp_n)|\ka(\pp_{n+1})$ may be trivial when $\pp_n$ is singular, in which case the equalities $[\ka(\pp):\ka(\pp_i)] = p^i$ in the proposition no longer hold.

We now study the primes of the separable base field extension $K^{sep}F|K^{sep}$. 

\begin{prop}\label{2023_02_06_13:00}
Let $\pp$ be a prime of $F|K$.
Then the number of the primes of $K^{sep} F | K^{sep}$ that lie over $\pp$ is equal to the separable degree $[\ka(\pp):K]_{sep}$ of the residue field extension $\ka(\pp)|K$.
Moreover, each such extended prime $\fq$ has degree $\deg(\fq) = [\ka(\pp):K]_{insep}$ and geometric singularity degree $\de(\fq)=\delta(\pp)/[\ka(\pp):K]_{sep}$. 
\end{prop}

\begin{proof}
Let $L$ be a finite separable extension of $K$, and let $\fq_1,\dots,\fq_r$ be the primes of $LF|L$ lying over $\pp$. Then
\[ [L:K] = \sum_{i=1}^r e_i f_i \]
where $e_i$ and $f_i$ are the ramification indices and the inertia indices of $\fq_i$ over $\pp$ respectively.
As the trace of $L{\cdot}F|F = (L \otimes_K F)|F$ is equal to $\mathrm{tr}_{L|K}\otimes \mathrm{id}_F$ where $\mathrm{tr}_{L|K}$ denotes the trace of the finite separable extension $L|K$, we deduce that the integral closure $\OO_{\fq_1} \cap \dots \cap \OO_{\fq_r}$ of $\OO_\pp$ in $L{\cdot} F$ is equal to $L\otimes_K \OO_\pp$ (i.e., the $L$-singularity degree of $\OO_\pp$ is zero).
It follows that the exponents of the Dedekind different of $L{\cdot} F | F$ are equal to zero, and so the ramification indices $e_i$ are equal to one.
It also follows that each residue field $\ka(\fq_i)$ ($i=1,\dots,r$) is generated by the images of $\ka(\pp)$ and $L$ inside it.

If $L$ contains the separable closure of $K$ in $\kappa({\pp})$ then $f_i = [L:K]/[\kappa({\pp}):K]_{sep}$ and $r = [\ka(\pp):K]_{sep}$, so in particular $\deg(\fq_i) = \deg(\pp)/r=[\ka(\pp):K]_{insep}$. 
Note that the equality $r = [\ka(\pp):K]_{sep}$ holds without assuming that $[L:K]$ is finite, and so it holds for $L = K^{sep}$. 
Since a similar remark applies to the identity $\deg(\fq_i) = [\ka(\pp):K]_{insep}$, we conclude that there are precisely $[\ka(\pp):K]_{sep}$ primes of $K^{sep} F|K^{sep}$ lying over $\pp$, each of degree $[\ka(\pp):K]_{insep}$.
It remains to compute their geometric singularity degrees. 
By the preceding paragraph, the $K^{sep}$-singularity degree of $\pp$ is zero. In light of \eqref{2023_02_22_20:30}, this means that the geometric singularity degree $\de(\pp)$ is equal to the $\ov K$-singularity degree of the semilocal ring $K^{sep} {\cdot} \OO_\pp$, which by \eqref{2023_02_22_20:25} is equal to the sum of the geometric singularity degrees of the primes of $K^{sep} F | K^{sep}$ lying over $\pp$.
But these primes are conjugate because the field extension $K^{sep} F|F$ is normal,
so it follows that their geometric singularity degrees coincide, that is, each of them has geometric singularity degree $\de(\pp)/[\ka(\pp):K]_{sep}$.
\end{proof}

\begin{cor}
The decomposition of the prime $\pp$ into $[\ka(\pp):K]_{sep}$ primes $\fq$ with $\deg(\fq)=[\ka(\pp):K]_{insep}$ and $\de(\fq) = \de(\pp) /[\ka(\pp):K]_{sep}$ already occurs in the extended function field $LF|L$, where $L$ is the separable closure of $K$ in $\ka(\pp)$.
\end{cor}

\begin{proof}
The proof of the proposition shows that there are exactly $[L:K]_{sep}$ primes $\fq$ in $LF|L$ above $\pp$, which have $\deg (\fq) = [L:K]_{insep}$.
Passing to the normal closure $L'$ of $L|K$, the proof also shows that the only prime $\fq'$ in $L'F|L'$ above a given prime $\fq$ has $\de(\fq') = \de(\pp)/[L:K]_{sep}$, and that $\fq$ has $L'$-singularity degree zero, i.e., $\de(\fq) = \de(\fq')$.
\end{proof}

\begin{cor}
For a prime $\pp$ the following assertions are equivalent
\begin{enumerate}[\upshape (i)]
    \item $\pp$ is separable,
    \item $\pp$ decomposes into rational primes in the extended function field $K^{sep} F | K^{sep}$,
    \item there is a finite separable extension field $L$ of $K$ such that $\pp$ decomposes into rational primes in $LF |L$.
\end{enumerate}
\end{cor}

\begin{cor}\label{2023_06_18_11:05}
The number of the primes of $\ov K F | \ov K$ lying over a prime $\pp$ is equal to the separable degree $[\ka(\pp):K]_{sep}$ of the residue field extension $\ka(\pp)|K$.
\end{cor}

\begin{proof}
Since primes are \emph{non-decomposed} in purely inseparable base field extensions, 
the number of the primes of $\overline{K}F|\overline{K}$ lying over $\pp$ concides with the number of the primes of $K^{sep} F | K^{sep}$ lying over $\pp$.
\end{proof}

\begin{cor}\label{2023_02_23_14:20}
Let $\pp$ be a prime of $F|K$.
Then $[\ka(\pp):K]_{sep}$ divides the geometric singularity degree $\de(\pp_n)$ and the $K^{1/p}$-singularity degree $\De_n=\De(\pp_n)$ for each non-negative integer $n$.
If $p>2$ then $[\ka(\pp):K]_{sep}$ also divides the integers $\frac 2 {p-1} \de(\pp_n)$ and $\frac 2 {p-1} \De(\pp_n)$.
\end{cor}

\begin{proof}
For every prime $\fq$ of $K^{sep}F|K^{sep}$ lying over $\pp$ we have $\de(\pp) = [\ka(\pp):K]_{sep} \, \de(\fq)$, where both $\de(\pp)$ and $\de(\fq)$ are divisible by $\frac{p-1}{2}$ when $p>2$ (see Proposition~\ref{2022_07_17_12:05}).
Analogous statements hold for the restricted primes $\pp_n$.
Note now that each $\pp_n$ has separable degree $[\ka(\pp_n):K]_{sep}=[\ka(\pp):K]_{sep}$ and $K^{1/p}$-singularity degree $\De_{n} = \de(\pp_n) - \de(\pp_{n+1})$.
\end{proof}

We say that a prime $\pp$ of $F|K$ is \textit{decomposed} if it is decomposed in the constant field extension $\ov K F | \ov K$, i.e., if there is more than one prime of $\ov K F|\ov K$ lying over $\pp$,
i.e., if its local ring $\OO_\pp$ is not geometrically unibranch.
For an example of a decomposed singular prime we refer to \cite[Example~2.5]{Sal14}.

\begin{cor}\label{2023_03_20_10:40}
    For a prime $\pp$ the following assertions are equivalent
    \begin{enumerate}[\upshape (i)]
        \item \label{2022_10_05_14:00} $\pp$ is non-decomposed,
        \item \label{2022_10_05_14:05} the residue field extension $\ka(\pp)|K$ is purely inseparable,
        \item \label{2022_10_05_14:10} there is an integer $n\geq0$ such that the prime $\pp_n$ is rational.
    \end{enumerate}
\end{cor}

Due to the second condition, non-decomposed points are also called \emph{purely inseparable} or \emph{perfect} in the literature \cite{BL22,Ros21}.

\begin{proof}
The equivalence between \ref{2022_10_05_14:00} and \ref{2022_10_05_14:05} follows immediately from Corollary~\ref{2023_06_18_11:05}. 
We note that $\pp$ is non-decomposed if and only if for some (and any) integer $n \geq 0$ the prime $\pp^{(n)}$, and therefore the prime $\pp_n$, is non-decomposed. By Proposition~\ref{2022_01_20_15:00} there is an integer $n \geq 0$ such that $\pp_n$ is separable. Clearly, a separable prime is purely inseparable if and only if it is rational.
\end{proof}

\begin{cor}\label{2023_20_02_16:05} 
A prime $\pp$ is rational if and only if it is separable and non-decomposed.
\end{cor}

In general, it may be hard to decide whether a given prime is non-decomposed.
The corollary below,
which follows immediately from Corollary~\ref{2023_02_23_14:20},
provides a sufficient criterion for a singular prime to be non-decomposed.

\begin{cor}\label{2022_03_04_16:40}
    If $p = 2$ and the integers $\De_0$, $\De_1$, $\De_2,\dots$ are coprime, then the prime $\pp$ is non-decomposed.
    Likewise, 
    if $p>2$ and the integers $\frac 2{p-1}\De_0$, $\frac 2{p-1}\De_1$, $\frac 2{p-1}\De_2,\dots$ are coprime, then $\pp$ is non-decomposed.
\end{cor}

Specializing Proposition~\ref{2023_01_25_19:00} to the non-decomposed case we get

\begin{prop}\label{2021_09_25_17:04}
Let $\pp$ be a non-singular non-decomposed prime of $F|K$, so in particular $\ka(\pp)|K$ is purely inseparable, say of degree $p^m$. Then $m$ is the smallest integer such that $\pp_m$ is rational.
\end{prop}

\begin{rem}\label{2023_05_31_20:55}
Let $C|K$ denote the regular geometrically integral curve associated to the function field $F|K$.
\begin{enumerate}[(i)]
\item 
Over each non-decomposed prime $\pp$ there lies a unique closed point in the extended curve $\ov C | \ov K := C\otimes_K \ov K | \ov K$, 
i.e., there is a unique point $x\in \ov C$ that is mapped to $\pp$ under the natural morphism $\ov C \to C$. The geometric singularity degree $\de(\pp)$ of $\pp$ coincides with the $\delta$-invariant $\de(\ov C, x)$ of $\ov C$ at $x$ as defined in \cite[p.\,69]{IIL20}.

\item
By Proposition \ref{2023_02_06_13:00}, over each prime $\pp$ in $F|K$ there are exactly $[\ka(\pp):K]_{sep}$ (non-decomposed) primes in $K^{sep} F | K^{sep}$, each of singularity degree $\de(\pp)/[\ka(\pp):K]_{sep}$. 
In other words, for each prime $\pp$ there are precisely $[\ka(\pp):K]_{sep}$ closed points  $x_i \in \ov C$ ($1 \leq i \leq [\ka(\pp):K]_{sep}$) that are mapped to $\pp$ by the morphism $\ov C \to C$, each of $\delta$-invariant $\de(\ov C,x_i)=\de(\pp)/[\ka(\pp):K]_{sep}$.

\item
Since singularity degrees are divisible by $(p-1)/2$ (see Proposition~\ref{2022_07_17_12:05}) we deduce that the $\de$-invariants $\de(\ov C,x)$ of the curve $\ov C$ are all multiples of $(p-1)/2$. In particular, they cannot be strictly smaller than $(p-1)/2$ unless $C$ is smooth.
This provides a new proof of the smoothness criterion in \cite[Theorem~5.7]{IIL20}.

\item
We also deduce that the singularities of $\ov C$ are unibranch. In other words, over each singular point $x\in \ov C$ there lies a unique point on the normalization of $\ov C$.
\end{enumerate}
\end{rem}

We now address the question raised after Proposition~\ref{2023_01_25_19:00}.
Given a singular prime $\pp$ we ask for a specific integer $n$ such that the restriction $\pp_n$ is separable.
To get an answer we work with the partitions of the geometric singularity degree $\de(\pp)$ as the sum of the $K^{1/p}$-singularity degrees $\De_i$, as indicated in Proposition~\ref{2022_07_17_12:05}.

Let $d$ be a positive integer. We consider the partitions
\[ d = d_1 + \dots + d_s \]
of $d$ by positive integers $d_i$ satisfying
\[ d_{i+1} \leq p^{-1} d_i \; \text{ for each $i = 1,\dots,s-1$}. \]
We define
\[ \tau_p(d) := \max\{ s + \min \{ v_p(d_1), \dots, v_p(d_s)\} \}, \]
where the maximum is taken over all such partitions and $v_p(d_i)$ denotes the exponent of the largest $p$-power that divides $d_i$.

\begin{prop}\label{2021_09_01_15:39}
    Let $\pp$ be a singular prime of $F|K$.
    Then the restricted prime $\pp_n$ is separable for all $n \geq \tau_p \big( \frac{2}{p-1} \frac {\de(\pp)}{[\ka(\pp):K]_{sep}} \big) $.
\end{prop}

Note that according to Corollary~\ref{2023_02_23_14:20} the integer $\frac 2 {p-1} \de(\pp)$ is divisible by $[\ka(\pp):K]_{sep}$, and so $\frac{2}{p-1} \frac {\de(\pp)}{[\ka(\pp):K]_{sep}}$ is indeed an integer.

\begin{proof}
    We take $d:=\frac{2}{p-1} \frac {\de(\pp)}{[\ka(\pp):K]_{sep}}$ and $d_i := \frac{2}{p-1} \frac {\De(\pp_{i-1})}{[\ka(\pp):K]_{sep}}$.
    Let $s$ be the largest integer such that $d_s>0$, that is, $\De(\pp_{s-1})\neq 0$ but $\De(\pp_s) = 0$, i.e., $\pp_{s-1}$ is singular but $\pp_s$ is non-singular.
    Let $m=\log_p [\ka(\pp_s):K]_{insep}$. By Proposition~\ref{2023_01_25_19:00}, the prime $\pp_{s+m}$ is separable.
    By Proposition~\ref{2021_09_24_16:03}, for each $i=1,\dots,s$ the degree $\deg (\pp_i)$ is a divisor of $[\ka(\pp):K]_{sep} \cdot d_i$.
    Because $\deg(\pp_s) = [\ka(\pp):K]_{sep} \cdot p^m$ is a divisor of each $\deg(\pp_i)$ we conclude
    $m\leq \min \{ v_p(d_1), \dots, v_p(d_s)\}$.
\end{proof}

To get the desired bound on $n$ so that $\pp_n$ is separable it remains to solve a combinatorics problem, namely, we must determine the precise value of $\tau_p(d)$. 
As this will depend on whether $d$ is a sum of consecutive $p$-powers
we introduce the following notation: for a pair of non-negative integers $j\leq i$ we write
\[ P^i_j := p^j + \dots + p^i = \sum_{r=j}^i p^r = \frac{p^{i+1}-p^j}{p-1}. \]
Note that for every positive integer $i$ the following inequalities hold
\[ P_0^{i-1} < P_i^i < P_{i-1}^i < \cdots < P_0^i. \]

\begin{prop}\label{2022_03_03_17:40}
    Let $d$ be a positive integer.
    If 
    $ P_0^{i-1} \leq d \leq P_0^i $,
    then
    \[ \tau_p(d) = 
    \begin{cases}
        i+1 & \text{if $d=P^i_j$ for some $j \leq i$,} \\ 
        i & \text{otherwise}. \\
    \end{cases} 
    \]
    Equivalently,
    \[ \tau_p(d) = 
        \begin{cases}
            \lc \log_p ((p-1) d + 1) \rc & \text{if $d$ is a sum of consecutive $p$-powers,} \\
            \lc \log_p ((p-1) d + 1) \rc - 1 & \text{otherwise.}
        \end{cases} 
    \]
\end{prop}

A straightforward consequence is the identity
\[ \tau_p(pd) = \tau_p(d) + 1. \]

\begin{proof}
The partition $d = ((d-P^{i-1}_0) + p^{i-1}) + p^{i-2} + \dots + 1$ shows that $\tau_p(d) \geq i$ whenever $d \geq P_0^{i-1}$. Moreover, if $d = P_j^i$ for some $j \leq i$ then $\tau_p(d) > i$, as follows from the partition $d = p^i + \dots + p^j$. 
Thus it suffices to show that
if $d \leq P_0^i$ then any partition $d = d_1 + \dots + d_s$ 
with $(d_1,\dots,d_s) \neq (p^{i},\dots, p^{i+1-s})$ satisfies 
\[ s + \min\{v_p(d_1),\dots,v_p(d_s)\} \leq i. \]
We argue by induction on $i$. The base case $i=1$ is clear, so we assume $i>1$. We may suppose that $s>1$; indeed, if $s=1$ then $v_p(d_1)<i$ because $d_1 \leq P_0^i$ and $d_1 \neq p^i$. 
As  
\[ d_2 + \dots + d_s \leq \frac {d_1} {p} + \dots + \frac{d_1} {p^{s-1}} < \frac{d_1} {p-1} = \frac{d-d_2- \dots -d_s}{p-1}, \] 
hence $d_2 + \dots + d_s  < p^{-1} d \leq p^{-1} P_0^i = p^{-1} + P_0^{i-1}$ and therefore 
$ d_2 + \dots + d_s \leq P_0^{i-1}$, 
it follows from the induction hypothesis that either $(d_2,\dots,d_s) = (p^{i-1},\dots, p^{i+1-s})$ or
\[ s - 1 + \min \{ v_p(d_2),\dots,v_p(d_s) \} \leq i-1. \]
In the second case the claim follows. In the first case it remains to show that $d_1$ is not a multiple of $p^{i+1-s}$. This holds because on the one hand $d_1 \geq  p d_2 = p^i, d_1 \neq p^i$  and therefore $d_1 - p^i > 0$, while on the other hand 
\[ d_1 - p^i = d - p^i - d_2 - \dots - d_s = d - P_{i+1-s}^i \leq P_0^i -P_{i+1-s}^i = p^0 + \dots + p^{i-s} < p^{i-s+1}. \qedhere \]
\end{proof}

A combination of Propositions~\ref{2021_09_01_15:39} and~\ref{2022_03_03_17:40} yields the desired bound on $n$ so that the prime $\pp_n$ is separable.
This depends on the characteristic $p>0$ and on the geometric singularity degree $\de(\pp)$ of the singular prime $\pp$.
In particular, when $\pp$ is non-decomposed we obtain a bound on $n$ so that $\pp_n$ is rational, thus answering the question raised before Proposition~\ref{2022_07_17_12:05}.

\begin{thm}\label{2022_03_04_12:45}
    Let $F|K$ be a one-dimensional separable function field of characteristic $p>0$.
    For a singular prime $\pp$ the following assertions hold.
    \begin{enumerate}[\upshape (i)]
        \item 
        The restriction $\pp_n$ of $\pp$ to the $n$th Frobenius pullback $F_n|K = F^{p^n}{\cdot}K|K$ is separable for all $n \geq \log_p \big( 2 \, \frac{\de(\pp)}{[\ka(\pp):K]_{sep}} + 1 \big)$; moreover, if the integer $\frac 2{p-1} \frac{\de(\pp)}{[\ka(\pp):K]_{sep}}$ is not a sum of consecutive $p$-powers then $\pp_n$ is separable for all 
        $n \geq \log_p \big( 2 \, \frac{\de(\pp)}{[\ka(\pp):K]_{sep}} + 1 \big) - 1$.
        \item \label{2023_04_03_23:55}
        Assume in addition that the prime $\pp$ is non-decomposed. 
        Then $\pp_n$ is rational for all $n \geq \log_p \big( 2  {\de(\pp)} + 1 \big)$; moreover, if the integer $\frac 2{p-1} {\de(\pp)}$ is not a sum of consecutive $p$-powers then $\pp_n$ is rational for all 
        $n \geq \log_p \big( 2{\de(\pp)} + 1 \big) - 1$.
    \end{enumerate}
\end{thm}

In the special case where $p>2$ and $\de(\pp) = p(p-1)/2$, the bound in~\ref{2023_04_03_23:55} is equal to $2$. This was obtained previously by Salomão \cite[Corollary~3.3]{Sal11}.

Let us look at the simplest example of the situation we are discussing.
Let $F|K$ be quasi-elliptic, i.e., suppose that $F|K$ and its extension $\ov K F | \ov K$ have genera $g=1$ and $\ov g = 0$.
By the genus change formula~\eqref{2022_03_04_18:10} there is a unique singular prime $\pp$, which has $\de(\pp)=1$ and $\De_0 = 1$, $\De_1=0$. 
In particular $p\leq 3$ (see Proposition~\ref{2022_07_17_12:05}). 
Also, the restricted prime $\pp_1$ is non-singular and the first Frobenius pullback $F_1|K$ is conservative of genus $g_1=0$.
By Corollary~\ref{2022_03_04_16:40} the prime $\pp$ is non-decomposed. 
Then Theorem~\ref{2022_03_04_12:45} implies that the restricted prime $\pp_2$ (resp. $\pp_1$) is rational if $p=2$ (resp. $p=3$),
and in turn the Frobenius pullback $F_2|K$ (resp. $F_1|K$) is a rational function field.
As explained in Section~\ref{2024_07_01_23:55}, it is then possible to add generators to $F_2$ (resp. $F_1$) through the Bedoya--St\"ohr algorithm to obtain a presentation of $F|K$, thus recovering Queen's characterization of quasi-elliptic function fields \cite{Queen71} (see \cite[Section~2]{HiSt23} for details).

For non-hyperelliptic function fields $F|K$ of genera $g=3$, $\ov g = 0$ in characteristic $p=2$ one can show that the first Frobenius pullbacks $F_1|K$ are quasi-elliptic.
Then the addition of a generator leads to a full characterization of these function fields (see \cite{HiSt23,HiSt24}).

\medskip

In the remaining of this section we show that the bound in Theorem~\ref{2022_03_04_12:45}~\ref{2023_04_03_23:55} is sharp in characteristic $p=2$.
In characteristic $p>2$, however, an analogous statement is false (see \cite{Hi24}).

\begin{prop}\label{2022_07_08_15:40}
    The bound provided by Theorem~\ref{2022_03_04_12:45}~\ref{2023_04_03_23:55} is sharp in characteristic $p=2$.
\end{prop}

In order to prove the proposition, we must construct for every positive integer $d$ a non-decomposed prime $\pp$ of geometric singularity degree $\de(\pp)=d$ whose restriction $\pp_{n-1}$ is non-rational,
where $n$ is the smallest integer allowed by the bound in Theorem~\ref{2022_03_04_12:45}~\ref{2023_04_03_23:55}, i.e.,
\[ n = \tau_2(2d) = \tau_2(d) + 1
 = \begin{cases}
    i+2 & \text{if $d=P^i_j$ for some $j \leq i$,} \\
    i+1 & \text{if $P^{i-1}_0 < d < P^i_0$ and $d \neq P^i_j$ for all $j\leq i$.}
\end{cases}  
\]
In Example~\ref{2022_07_09_18:00} below we build for every $i > j \geq 0$ and every $\ell\geq 0$ a non-decomposed prime $\pp$ of geometric singularity degree
\[ \de(\pp) = P_{j}^{i} + \ell \cdot 2^{j+1} \]  
with the property that $\pp_{i+1}$ and $\pp_{i+2}$ are non-rational and rational respectively.
Similarly, in Example~\ref{2022_07_09_18:05} we construct for every $i \geq 0$ a non-decomposed prime $\pp$ of geometric singularity degree
\[ \de(\pp) = 2^i = P^i_i \]
with the property that $\pp_{i+1}$ and $\pp_{i+2}$ are non-rational and rational respectively.
Before getting to the examples themselves we show how the proposition is obtained from them.

\begin{proof}[Proof of Proposition~\ref{2022_07_08_15:40}]
In view of the two examples,
it is enough to show that if $d$ is a positive integer such that $P^{i-1}_0 < d < P^i_0$ and $d\neq P^i_j$ for all $j\leq i$, 
then it can be written as $d=P_j^{i-1} + \ell \cdot 2^{j+1}$ for some $j < i-1$ and some $\ell\geq 0$.
Indeed, if this were not the case then
\[ d\not\equiv 2^{j} \pmod {2^{j+1}} \quad \text{for each $j=0,\dots,i-2$,} \]
which means $d\equiv 0 \pmod {2^{i-1}}$, and therefore $d \in \{ P^i_{i}, P^i_{i-1}\}$ as $P^{i-1}_0 < d < P^i_0$, a contradiction.
\end{proof}

\begin{example}\label{2022_07_09_18:00}
Let $i>j \geq 0$ and $\ell\geq0$.
We construct a non-decomposed prime $\pp$ of geometric singularity degree
\[ \de(\pp) = P_{j}^{i} + \ell \cdot 2^{j+1} \]  
with the property that $\pp_{i+1}$ and $\pp_{i+2}$ are non-rational and rational respectively.
Consider the function field $F|K=K(y,u)|K$ in characteristic $p=2$ given by the following relation
    \[ (a + z^{2^{j+1}}) z + y^{2^{i-j}} = 0, \]
    where $z:= u^2 + y^{1+2\ell}$ and $a\in K\setminus K^2$.
    Then $y^{2^{i-j}} = (a + z^{2^{j+1}}) z$ and $u^2 = z + y^{1+2\ell}$,
    whence the Frobenius pullbacks of $F|K$ take the form
    \[ F_n|K =
    \begin{cases}
        K(y,u)|K & \text{if $n=0$}, \\        
        K(z,y^{2^{n-1}})|K & \text{if $0 < n < i-j+1$}, \\
        K(z)|K & \text{if $n=i-j+1$}.
    \end{cases}
    \]
    Let $\pp$ be the zero of the function $z^{2^{j+1}}+a$, i.e., let $\pp$ be the only prime of $F|K$ such that $v_\pp(z^{2^{j+1}}+a)>0$.
    Then the restricted prime $\pp_{i-j+1}$ is the $(z^{2^{j+1}}+a)$-adic prime of the rational function field $F_{i-j+1}|K=K(z)|K$, i.e., $v_{\pp_{i-j+1}}(z^{2^{j+1}}+a)=1$, 
    hence it is non-singular with residue field $\ka(\pp_{i-j+1}) = K(a^{1/2^{j+1}})$ and degree $\deg(\pp_{i-j+1}) = 2^{j+1}$.
    By Corollary~\ref{2023_03_20_10:40} this prime is non-decomposed, and 
    by Proposition~\ref{2021_09_25_17:04} its restrictions $\pp_{i+1}$ and $\pp_{i+2}$ are non-rational and rational respectively.
    The function $x:=z^{2^{j+1}}+a \in F_{i+2}$ satisfies $F_{i+2}=K(x)$, and moreover it is a local parameter at the rational prime $\pp_{i+2}$.
    
    We compute the geometric singularity degree of the non-decomposed prime $\pp$ by applying the algorithm developed in \cite{BedSt87}.
    Because $y\in F_1$ and $y^{2^{i-j}} = x z \in F_{i-j+1}$ is a local parameter at the prime $\pp_{i-j+1}$, for every $0 < n < i-j+1$ the prime $\pp_n$ is ramified over $F_{n+1}$, i.e., $\deg(\pp_n)=\deg(\pp_{n+1})$, 
    and the function $y^{2^{n-1}} \in F_n$ is a local parameter at $\pp_n$. 
    As the differential $d(y^{2^{n-1}})^{2^{i+2-n}} = dy^{2^{i+1}}= x^{2^{j+1}}dx$ of $F_{i+2}|K=K(x)|K$ has order $2^{j+1}$ at $\pp_{i+2}$, this implies by \cite[Theorem~2.3]{BedSt87} that
    \[ \de(\pp_n) = 2 \de(\pp_{n+1})  + \tfrac12 v_{\pp_{i+2}} (dy^{2^{i+1}}) = 2 \de(\pp_{n+1}) + 2^{j} \qquad (0 < n < i-j+1). \]
    Note now that the residue classes $z(\pp),y(\pp), u(\pp) \in \ka(\pp)$ satisfy $y(\pp) = 0$, $z(\pp) = a^{1/2^{j+1}}$, $u(\pp)^2 = z(\pp)$, and $\ka(\pp_1)=K(z(\pp))$.
    Since $u(\pp) = z(\pp)^{1/2}$ 
    does not lie in $\ka({\pp_1})$ the prime $\pp$ is unramified over $F_1$, so it follows from \cite[Theorem~2.3]{BedSt87} that
    \[ \de(\pp) = 2 \de(\pp_1) + \tfrac12 v_{\pp_{i+2}} (du^{2^{i+2}}) = 2\de(\pp_1) + 2^{j} + \ell \cdot 2^{j+1}, \]
    where the last equality is due to the fact that the differential $du^{2^{i+2}} = x ^{2^{j+1}(1+2\ell)} (a + x)^{2\ell} dx$ of $F_{i+2}|K$ has order $2^{j+1}(1+2\ell)$ at $\pp_{i+2}$.
    This shows that $\pp$ has the desired geometric singularity degree.
\end{example}

\begin{example}\label{2022_07_09_18:05}
Let $i \geq 0$. 
We construct a non-decomposed prime $\pp$ of geometric singularity degree
$ \de(\pp) = 2^i$,
with the property that $\pp_{i+1}$ and $\pp_{i+2}$ are non-rational and rational respectively.
Let $F|K = K(z,y)|K$ be the function field in characteristic $p=2$ defined by the equation
\[ y^{2} = (a+z^{2^{i+1}}) z,  \]
where $a\in K\setminus K^2$.
The first Frobenius pullback is equal to
\[ F_1|K = K(z)|K. \]
Let $\pp$ be the zero of the function $z^{2^{i+1}}+a$, so that its restriction $\pp_1$ is the $(z^{2^{i+1}}+a)$-adic prime of the rational function field $F_1|K=K(z)|K$, i.e., $v_{\pp_1}(z^{2^{i+1}}+a)=1$.
This implies that $\pp_1$ is a non-singular prime of degree $\deg(\pp_1) = 2^{i+1}$
and that the primes $\pp_{i+2}$ and $\pp_{i+1}$ are rational and non-rational respectively.
As $y^2 = xz$ is a local parameter at $\pp_1$ we conclude
$ \de(\pp) = 2 \de(\pp_{1}) + \frac12 v_{\pp_{i+2}} (dy^{2^{i+2}}) = 2^{i}. $
\end{example}

\section{A pencil of singular quartics in characteristic 2}\label{2022_05_06_17:00}
In this section we study the geometry of a fibration by singular rational plane projective quartics over the projective line in characteristic $2$.
The generic fibre of this fibration has a singular non-decomposed prime $\pp$ of geometric singularity degree $\de(\pp) = 3$, with the property that its restriction $\pp_2$ to the second Frobenius pullback is non-rational.
This means that the generic fibre attains the bound provided by Theorem~\ref{2022_03_04_12:45}~\ref{2023_04_03_23:55} for $p=2$ and $\de(\pp)=3$.
We determine as well the minimal regular model of the fibration.

\medskip
Let $k$ be an algebraically closed field of characteristic $p=2$.
Consider the integral
projective algebraic surface over $k$
\[ S \su \PP^2\times \PP^1 \]
cut out by the bihomogeneous polynomial equation
\begin{equation}\label{2022_05_04_16:10}
    T_0(Z^4 + X^2 Y^2 + X^3 Z) + T_1 (Y^4 + X^2 Z^2)=0,
\end{equation}
where $X,Y,Z$ and $T_0,T_1$ represent the homogenous coordinates of $\PP^2$ and $\PP^1$ respectively.
The surface $S$ has a unique singular point, namely $P=((1:0:0),(0:1))$, as follows from the Jacobian criterion.
The second projection
\[ \phi:S\tto \PP^1, \]
which is proper and flat \cite[Chapter~III, Proposition~9.7]{Har77}, 
yields a fibration by plane projective quartic curves over $\PP^1$.
The fibre over each point of the form $(1:c)$ in $\PP^1$ is isomorphic to the plane projective quartic curve $S_c$ cut out by the equation
\[ Z^4 + X^2 Y^2 + X^3 Z + c (Y^4 + X^2 Z^2) = 0, \]
which has a unique singular point at
\[ P_c:=(0:1:c^{1/4}). \]
This curve is rational and integral, 
and its arithmetic genus is equal to $3$, as follows from the genus-degree formula for plane curves.
The singular point $P_c$ is unibranch of singularity degree $3$ and multiplicity~$2$ (if $c^3\neq 1$) or~$3$ (if $c^3=1$),
and its tangent line 
cuts the quartic curve only at $P_c$.
The quartic curve is \emph{strange}, that is,
all its tangent lines pass through the unique common point $(0:1:0)$.
If $c=0$, then this point coincides with the singular point $P_c$, and so each tangent line at a non-singular point intersects the curve at two points but is not a bitangent. 
In the opposite case $c\neq 0$, every such tangent line is a bitangent.

In analogy to the theory of elliptic curves, we note that the curve $S_c$ is \emph{homogeneous}, that is, for any two non-singular points there is an automorphism that maps the first point to the second one. 
Indeed, given a non-singular point $(x_0:y_0:z_0)\in S_c$, the projective transformation
\[ (x:y:z) \mapstto (x_0 x : x_0 y + y_0 x : x_0 z + z_0 x) \]
defines an automorphism of $S_c$ mapping $(x_0:y_0:z_0)$ to the point $(1:0:0)$.

Over the point $(0:1)$ of the base $\PP^1$ the fibre of $\phi$ degenerates to the non-reduced curve
\[ (Y^2 + X Z)^2 = 0. \]
This is the \textit{bad fibre} of the fibration in the sense that its behaviour differs from the generic behaviour of the fibres.

The fibration $\phi : S \to \PP^1$ has a section, namely the horizontal curve $(1:0:0) \times \PP^1 \su S$.
Also, the \emph{non-smooth locus}, which comprises the singular points $((0:1:c^{1/4}),(1:c))$ on the fibres, is the curve in $\PP^2 \times \PP^1$ defined by the bihomogeneous polynomial equations
\[ X=0 \quad \text{and} \quad T_0 Z^4 + T_1 Y^4 = 0.   \]
This is a rational curve, which cuts the bad fibre $\phi^{-1}(0:1)$ at the point $((0:0:1),(0:1))$.
It is mapped onto the base $\PP^1$ according to $(y:z) \mapsto (y^4:z^4)$, and so it is a purely inseparable cover of degree $4$ of $\PP^1$. (This will also follow from Proposition~\ref{2023_01_30_21:00} below.)

The generic fibre $C$ of the fibration $\phi:S\to \PP^1$ is the quartic curve over the function field $K = k(t) := k(T_1/T_0)$ of the base $\PP^1$ defined by the homogeneous equation~\eqref{2022_05_04_16:10}.
Its function field $F:=K(C)$ coincides with the function field $k(S)$ of the total space $S$. Dehomogenizing $X \mapsto 1$ and $T_0 \mapsto 1$ in equation~\eqref{2022_05_04_16:10} we obtain 
\[  F = k(S) = k(t,y,z) = K(y,z) \]
where the affine coordinate functions $t$, $y$ and $z$ of the surface $S$ satisfy the equation
\begin{equation}\label{2022_03_04_18:15}
    (z^4 + y^2 + z) + t (y^4 + z^2) = 0.
\end{equation}
The function field $F|K=K(y,z)|K$ of $C$ is therefore generated over $K=k(t)$ by the functions $y$ and $z$, which satisfy equation~\eqref{2022_03_04_18:15}.
The following proposition lists some properties of the generic fibre $C$ and its singular primes.

\begin{prop}\label{2023_01_30_21:00}
The regular curve $C$ has arithmetic genus $h^1(C,\OO_C) = 3$. 
The genus of the normalization of its extension $C\otimes_K \overline{K}$ is equal to zero. 
Furthermore, there is a unique singular prime $\pp$ in $C$, which is non-decomposed and satisfies
\begin{enumerate}[\upshape (i)]
    \item $\delta(\pp)=3$, $\delta(\pp_1)=1$, and $\delta(\pp_n)=0$ for each $n \geq 2$; 
    \item $\deg(\pp)=4$, $\deg(\pp_1)=\deg(\pp_2)=2$, and $\deg(\pp_n)=1$ for each $n \geq 3$.
\end{enumerate}
\end{prop}

In particular, the prime $\pp$ attains the bound in Theorem~\ref{2022_03_04_12:45}~\ref{2023_04_03_23:55} for $p=2$ and $\de(\pp)=3$.

\begin{proof}
By the genus-degree formula for plane curves, the curve $C$ has arithmetic genus $h^1(C,\OO_C) = 3$; equivalently, the function field $F|K$ has genus $g=3$.
Consider the function $u:= z + y^2$ in $F=K(y,z)$, and notice that it satisfies the relation
\begin{equation*}
    tu^2 + u = z^4.
\end{equation*}
The first three iterated Frobenius pullbacks of $F|K$ are then given by
\[ F_1|K = K(u,z)|K, \quad F_2|K = K(u,z^2)|K, \quad F_3|K = K(u)|K. \]
As the latter function field is rational, so is the extended function field $\ov K F|\ov K$, i.e., the extended curve $C \otimes_K \ov K$ is rational and its normalization has genus $\ov g = 0$.

Let $\pp$ denote the only pole of $u$, i.e., let $\pp$ be the only prime of $F|K$ such that $v_\pp(u)<0$.
It is non-decomposed because its restriction $\pp_3$ to $F_3|K$ is a rational prime with local parameter $u^{-1}\in F_3$.
In particular, we can determine its invariants by applying the algorithm in \cite{BedSt87}.
Since the function $(z^2 u^{-1})^2 + t = u^{-1}$ belongs to the maximal ideal of the local ring $\OO_{\pp_3}$, the value $(z^2 u^{-1})(\pp_2) = t^{1/2}$ lies outside $\ka(\pp_3)=K$.
Thus the prime $\pp_2$ of $F_2|K$ is unramified over $F_3$ with residue field $\ka({\pp_2}) = K(t^{1/2})$, and therefore $\de(\pp_2) = \frac12 v_{\pp_3} (d (z^2 u^{-1})^2 \big) = 0$ by \cite[ Theorem~2.3]{BedSt87}.
As the function $zu^{-1} \in F_1$ has fourth power $(zu^{-1})^4 = tu^{-2} + u^{-3}$, it follows that the prime $\pp_1$ of $F_1|K$ is ramified over $F_2$ with local parameter $zu^{-1}$, and thus $\de(\pp_1) = \frac12 v_{\pp_3} (d(zu^{-1})^4) = 1$ by \cite[ Theorem~2.3]{BedSt87}.

It remains to determine the invariants of $\pp$.
Note that $\de(\pp)=3$, since on the one hand $\de(\pp)\leq g=3$, while on the other hand $\De_0 \geq 2 \De_1=2$ (see Proposition~\ref{2024_07_14_19:10}).
Because $g - \ov g= 3$, it follows from Rosenlicht's genus drop formula~\eqref{2022_03_04_18:10} that $\pp$ is the only singular prime of $F|K$.
Moreover, as the function $(\frac zy)^8 + t^2 \in F_3$ has order $2$ at $\pp_3$, the value $(\frac zy)(\pp)=t^{1/4}$ does not belong to $\ka(\pp_1)=K(t^{1/2})$. 
This proves that $\pp$ has residue field $\ka(\pp)=K(t^{1/4})$ and degree $\deg(\pp)=4$.
\end{proof}

By a theorem of Lichtenbaum--Shafarevich,
a (relatively) minimal regular model of the fibration $S\to \PP^1$ 
exists and 
is unique up to isomorphism
(see \cite[Chapter~9, Theorem~3.21 and Corollary~3.24]{Liu02}).
In general, it is difficult to unveil the structure of such a minimal model, but here we can achieve an explicit description by performing blowups, as described below. 
Similar results for families of curves on rational normal scrolls were established in \cite{St04}.

Local computations show that the only singular point $P=((1:0:0),(0:1))$ of the surface $S$ is a rational double point of type $A_{15}$, which is resolved by blowing up the surface eight times.
This in turn gives a smooth surface $\wt S$ and a new proper flat 
fibration
\[ f:\wt S \tto S \overset\phi\tto \PP^1 \]
whose fibers over the points $(1:c)$ of $\PP^1$ coincide with those of $\phi$.
Over the point $(0:1)$ the fibre $f^*(0:1)$ is given by a linear combination of smooth rational curves
\begin{equation}\label{2021_08_24_04:55}
    \begin{aligned}
    f^*(0:1) &= 2E + E_1^{(1)} + E_2^{(1)} + 2 E_1^{(2)} + 2 E_2^{(2)} + 3 E_1^{(3)} + 3 E_2^{(3)} + 4 E_1^{(4)} + 4 E_2^{(4)}  \\
    &\qquad + 5 E_1^{(5)} + 5 E_2^{(5)} + 6 E_1^{(6)} + 6 E_2^{(6)} + 7 E_1^{(7)} + 7 E_2^{(7)} + 8 E^{(8)},
    \end{aligned}
\end{equation}
which intersect transversely according to the Coxeter-Dynkin diagram in Figure~\ref{2020_11_25_23:56}.
In this diagram the vertex $E$ represents the strict transform of the support $\phi^{-1}(0:1)$ of the bad fibre, while
the dashed line means that the strict transform $H$ of the horizontal curve $(1:0:0)\times \PP^1 \su S$ intersects the fibre $f^*(0:1)$ transversely at the component $E_2^{(1)}$ but does not belong to $f^*(0:1)$.
\begin{figure}[h]
\centering
\begin{tikzpicture}[line cap=round,line join=round,x=0.7cm,y=0.7cm]
\draw [line width=1.2pt] (-7.,0.)-- (7.,0.);
\draw [line width=1.2pt, dashed] (7.,0.)-- (8.,0.);
\draw [line width=1.2pt] (0.,0.)-- (0.,-1.);
\begin{scriptsize}
\draw [fill=black] (-7,0) circle (2pt);
\draw[color=black] (-7,0.5) node {$E^{(1)}_1$};    
\draw [fill=black] (-6,0) circle (2pt);
\draw[color=black] (-6,0.5) node {$E^{(2)}_1$};
\draw [fill=black] (-5,0) circle (2pt);
\draw[color=black] (-5,0.5) node {$E^{(3)}_1$};
\draw [fill=black] (-4,0) circle (2pt);
\draw[color=black] (-4,0.5) node {$E^{(4)}_1$};
\draw [fill=black] (-3,0) circle (2pt);
\draw[color=black] (-3,0.5) node {$E^{(5)}_1$};
\draw [fill=black] (-2.,0.) circle (2pt);
\draw[color=black] (-2.,.5) node {$E^{(6)}_1$};
\draw [fill=black] (-1.,0.) circle (2pt);
\draw[color=black] (-1.,.5) node {$E^{(7)}_1$};
\draw [fill=black] (0.,0.) circle (2pt);
\draw[color=black] (0.,.5) node {$E^{(8)}$};
\draw [fill=black] (0.,-1.) circle (2pt);
\draw[color=black] (0.,-1.5) node {$E$};
\draw [fill=black] (1.,0.) circle (2pt);
\draw[color=black] (1.,0.5) node {$E^{(7)}_2$};
\draw [fill=black] (2.,0.) circle (2pt);
\draw[color=black] (2.,.5) node {$E^{(6)}_2$};
\draw [fill=black] (3.,0.) circle (2pt);
\draw[color=black] (3.,0.5) node {$E^{(5)}_2$};
\draw [fill=black] (4.,0.) circle (2pt);
\draw[color=black] (4,.5) node {$E^{(4)}_2$};
\draw [fill=black] (5.,0.) circle (2pt);
\draw[color=black] (5.,0.5) node {$E^{(3)}_2$};
\draw [fill=black] (6,0) circle (2pt);
\draw[color=black] (6,0.5) node {$E^{(2)}_2$};
\draw [fill=black] (7,0) circle (2pt);
\draw[color=black] (7,0.5) node {$E^{(1)}_2$};
\draw [fill=black] (8,0) circle (2pt);
\draw[color=black] (8,0.5) node {$H$};
\end{scriptsize}
\end{tikzpicture}
\caption{Dual diagram of the fibre $f^*(0:1)$}\label{2020_11_25_23:56}
\end{figure}
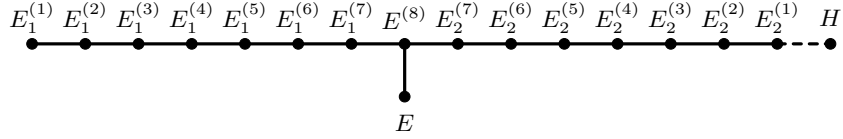

\noindent
Since a fibre meets its components with intersection number zero, equation~\eqref{2021_08_24_04:55} allows us to compute the self-intersection number of each component of $f^*(0:1)$. Thus
\[ E\cdot E = -4, \qquad E^{(i)}_j \cdot E^{(i)}_j = -2 \:\text{ for each $i,j$}, \quad E^{(8)} \cdot E^{(8)} = -2. \]
In particular, by Castelnuovo's contractibility criterion the smooth projective surface $\wt S$ is relatively minimal
over $\PP^1$,
and hence the fibration $\wt S \to \PP^1$ is the minimal regular model of the original fibration $S \to \PP^1$.

However, as we will see in a moment, the surface $\wt S$ is not relatively minimal as an algebraic surface over $\Spec(k)$.
In other words, it contains at least one smooth rational curve of self-intersection $-1$.
To see this in detail, we note that the first projection
\[ S \tto \PP^2 \]
is a birational morphism whose inverse
\[ \PP^2 \ttto S, \quad 
(x:y:z)\mapsto \big( (x:y:z),(y^4 + x^2 z^2 : z^4 + x^2 y^2 + x^3 z)\big) \]
is regular at all points of $\PP^2$ except $(1:0:0)$.
More precisely, the map $S \to \PP^2$ contracts the horizontal curve $(1:0:0)\times \PP^1 \su S$ to the point $(1:0:0)$ and induces an isomorphism between $S \setminus ((1:0:0)\times \PP^1) $ and $\PP^2 \setminus \{(1:0:0)\}$.
The composition $\wt S \to S \to \PP^2$ is then a birational morphism between smooth projective surfaces, which factors as the composition of finitely many blowups centered at smooth points. Its exceptional fibre, i.e., the fibre over $(1:0:0)$, comprises the sixteen smooth rational curves $E_1^{(1)},\dots , E_2^{(1)},H$, where $H$ is the strict transform under $\wt S \to S$ of the horizontal curve $(1:0:0)\times \PP^1\su S$  (see Figure~\ref{2020_11_25_23:56}).
This shows that the smooth projective surface $\wt S$ is rational, and that the smooth rational curve $H$ is contractible, i.e., it has self-intersection $-1$.

We collect our results on the geometry of the surface $\wt S$ in the following theorem.

\begin{thm}\label{2022_02_11_12:20}        
    The fibration $f:\wt S \to \PP^1$ is the minimal 
    regular
    model of the fibration $\phi: S\to \PP^1$.
    Its fibres over the points $(1:c)$ coincide with the corresponding fibres of $\phi$, while its fibre over the point $(0:1)$ is a linear combination of smooth rational curves as in \eqref{2021_08_24_04:55}, which intersect transversely according to the diagram in Figure~\ref{2020_11_25_23:56}.
    
    The smooth projective surface $\wt S$ is rational and the strict transform $H \su \wt S$ of the curve $(1:0:0)\times \PP^1\su S$ is a horizontal smooth rational curve of self-intersection $-1$.
    If we blow down successively the curves $H$, $E_2^{(1)}$, $E_2^{(2)}$, $\dots$, $E_1^{(2)}$ and $E_1^{(1)}$, then we obtain a surface isomorphic to the projective plane.
    
    The non-smooth locus of the fibration $f:\wt S \to \PP^1$ is a smooth rational horizontal curve, which is purely inseparable of degree $4$ over the base $\PP^1$.    
\end{thm}

By the variant of the Faltings-Mordell theorem for function fields \cite{Sam66,Vol91}, the generic fibre $C|K$ has only finitely many $K$-rational points. This means that the fibration $\wt S \to \PP^1$ has only finitely many horizontal prime divisors of degree 1 over the base $\PP^1$.

\begin{prop}
The generic fibre $C|K$ has only one $K$-rational point. The corresponding horizontal prime divisor of degree 1 is the contractible curve $H$.
\end{prop}

\begin{proof}
Let $\fq$ be the $K$-rational point of $C$ corresponding to the horizontal curve $H\su \wt S$, i.e., let $\fq$ be the only prime of $F|K$ such that the rational functions $z,y \in K(C)$ in \eqref{2022_03_04_18:15} satisfy $z(\fq) = y (\fq) = 0$. Clearly, the function $u:= z + y^2 \in F$ introduced in the proof of Proposition~\ref{2023_01_30_21:00} also satisfies $u(\fq) = 0$.

Seeking a contradiction, we assume that there is another rational prime $\fq'\neq \fq$. As follows from $z^4 = u + t u^2$ and $y^2 = z + u$, the value $u(\fq')\in K$ is non-zero, and so is the value $z(\fq') \in K$ because otherwise the equality $y(\fq')^2 + t u(\fq')^2 = z(\fq')^4 = 0$ contradicts $t\notin K^2$.
Thus there exist non-zero polynomials $f,g,F,G \in k[t]$ with $(f,g)=(F,G)=1$ satisfying the identity $(\frac FG)^4 = \frac fg + t (\frac fg)^2$, i.e.,
\[ F^4 g^2 = G^4 f (g + tf). \]
Since $G^4$ divides $g^2$, the polynomial $f$ is coprime with $G$ and therefore it is a fourth power in $k[t]$.
It follows that $g = G^2 g'$, $f= f'^4$ and $F=f' F'$ for some polynomials $f',g',F'$ in $k[t]$.
From the relation $F'^4 g'^2 = G^2 g' + t f'^4$ we deduce that $g'$ divides $t$, i.e., $g'$ is either a constant or a constant times $t$.
In light of $k=k^2$, both possibilities yield the contradiction $t\in K^2$.
\end{proof}


\begin{bibdiv}
\begin{biblist}
\bib{BedSt87}{article}{
  author={Bedoya, Hernando},
  author={St\"ohr, Karl-Otto},
  title={An algorithm to calculate discrete invariants of singular primes in function fields},
  journal={J. Number Theory},
  volume={27},
  date={1987},
  number={3},
  pages={310--323},
}

\bib{BM76}{article}{
  author={Bombieri, Enrico},
  author={Mumford, David},
  title={Enriques' classification of surfaces in char. $p$. III},
  journal={Invent. Math.},
  volume={35},
  date={1976},
  pages={197--232},
}

\bib{BL22}{article}{
  author={Bragg, Daniel},
  author={Lieblich, Max},
  title={Perfect points on curves of genus 1 and consequences for supersingular K3 surfaces},
  journal={Compos. Math.},
  volume={158},
  date={2022},
  number={5},
  pages={1052--1083},
}

\bib{FaSc20}{article}{
  author={Fanelli, Andrea},
  author={Schr\"oer, Stefan},
  title={Del Pezzo surfaces and Mori fiber spaces in positive characteristic},
  journal={Trans. Amer. Math. Soc.},
  volume={373},
  date={2020},
  number={3},
  pages={1775--1843},
}

\bib{EGA}{article}{
  author={Grothendieck, Alexander},
  title={\'El\'ements de g\'eom\'etrie alg\'ebrique (r\'edig\'es avec la collaboration de Jean Dieudonn\'e)},
  journal={Inst. Hautes Études Sci. Publ. Math.},
  date={{1960--1967}},
  number={4, 8, 11, 17, 20, 24, 28, 32},
  label={EGA},
}

\bib{Har77}{book}{
  author={Hartshorne, Robin},
  title={Algebraic geometry},
  series={Graduate Texts in Mathematics},
  volume={52},
  publisher={Springer-Verlag, New York-Heidelberg},
  date={1977},
  pages={xvi+496 pp.},
}

\bib{Hi24}{article}{
  author={Hilario, Cesar},
  title={Turning non-smooth points into rational points},
  date={2024},
  note={Preprint at \href {https://arxiv.org/abs/2402.14969}{\textsf {arXiv:2402.14969}}},
}

\bib{HiSc23}{article}{
  author={Hilario, Cesar},
  author={Schr\"oer, Stefan},
  title={Generalizations of quasielliptic curves},
  journal={\'{E}pijournal G\'{e}om. Alg\'{e}brique},
  volume={7},
  date={2023},
  pages={Art. 23, 31 pp},
}

\bib{HiSt23}{article}{
  author={Hilario, Cesar},
  author={St\"ohr, Karl-Otto},
  title={Fibrations by plane quartic curves with a canonical moving singularity},
  date={2023},
  note={Preprint at \href {https://arxiv.org/abs/2306.08579}{\textsf {arXiv:2306.08579}}},
}

\bib{HiSt24}{article}{
  author={Hilario, Cesar},
  author={St\"ohr, Karl-Otto},
  title={Fibrations by plane projective rational quartic curves in characteristic two},
  date={2024},
  note={Preprint},
}

\bib{IIL20}{article}{
  author={Ito, Kazuhiro},
  author={Ito, Tetsushi},
  author={Liedtke, Christian},
  title={Deformations of rational curves in positive characteristic},
  journal={J. Reine Angew. Math.},
  volume={769},
  date={2020},
  pages={55--86},
}

\bib{JW21}{article}{
  author={Ji, Lena},
  author={Waldron, Joe},
  title={Structure of geometrically non-reduced varieties},
  journal={Trans. Amer. Math. Soc.},
  volume={374},
  date={2021},
  number={12},
  pages={8333--8363},
}

\bib{Ko91}{article}{
  author={Koll\'ar, J\'anos},
  title={Extremal rays on smooth threefolds},
  journal={Ann. Sci. \'Ecole Norm. Sup. (4)},
  volume={24},
  date={1991},
  number={3},
  pages={339--361},
}

\bib{Lan79}{article}{
  author={Lang, William E.},
  title={Quasi-elliptic surfaces in characteristic three},
  journal={Ann. Sci. \'Ecole Norm. Sup. (4)},
  volume={12},
  date={1979},
  number={4},
  pages={473--500},
}

\bib{Liu02}{book}{
  author={Liu, Qing},
  title={Algebraic geometry and arithmetic curves},
  series={Oxford Graduate Texts in Mathematics},
  volume={6},
  publisher={Oxford University Press, Oxford},
  date={2002},
  pages={xvi+576 pp.},
}

\bib{Ma16}{article}{
  author={Maddock, Zachary},
  title={Regular del Pezzo surfaces with irregularity},
  journal={J. Algebraic Geom.},
  volume={25},
  date={2016},
  number={3},
  pages={401--429},
}

\bib{Mat80}{book}{
  author={Matsumura, Hideyuki},
  title={Commutative algebra},
  series={Mathematics Lecture Note Series},
  volume={56},
  edition={2},
  publisher={Benjamin/Cummings Publishing Co., Inc., Reading, MA},
  date={1980},
  pages={xv+313},
}

\bib{PW22}{article}{
  author={Patakfalvi, Zsolt},
  author={Waldron, Joe},
  title={Singularities of general fibers and the LMMP},
  journal={Amer. J. Math.},
  volume={144},
  date={2022},
  number={2},
  pages={505--540},
}

\bib{Queen71}{article}{
  author={Queen, Clifford S.},
  title={Non-conservative function fields of genus one. I},
  journal={Arch. Math. (Basel)},
  volume={22},
  date={1971},
  pages={612--623},
}

\bib{Ros52}{article}{
  author={Rosenlicht, Maxwell},
  title={Equivalence relations on algebraic curves},
  journal={Ann. of Math. (2)},
  volume={56},
  date={1952},
  pages={169--191},
}

\bib{Ros21}{article}{
  author={R\"{o}ssler, Damian},
  title={Purely inseparable points on curves},
  conference={ title={Abelian varieties and number theory}, },
  book={ series={Contemp. Math.}, volume={767}, publisher={Amer. Math. Soc., Providence, RI}, },
  date={2021},
  pages={89--96},
}

\bib{Sal11}{article}{
  author={Salom\~ao, Rodrigo},
  title={Fibrations by nonsmooth genus three curves in characteristic three},
  journal={J. Pure Appl. Algebra},
  volume={215},
  date={2011},
  number={8},
  pages={1967--1979},
}

\bib{Sal14}{article}{
  author={Salom\~ao, Rodrigo},
  title={Fibrations by curves with more than one nonsmooth point},
  journal={Bull. Braz. Math. Soc. (N.S.)},
  volume={45},
  date={2014},
  number={2},
  pages={267--292},
}

\bib{Sam66}{book}{
  author={Samuel, Pierre},
  title={Lectures on old and new results on algebraic curves},
  series={Tata Institute of Fundamental Research Lectures on Mathematics, No. 36},
  note={Notes by S. Anantharaman},
  publisher={Tata Institute of Fundamental Research, Bombay},
  date={1966},
  pages={ii+127+iii pp.},
}

\bib{Sc08}{article}{
  author={Schr\"oer, Stefan},
  title={Singularities appearing on generic fibers of morphisms between smooth schemes},
  journal={Michigan Math. J.},
  volume={56},
  date={2008},
  number={1},
  pages={55--76},
}

\bib{Sc09}{article}{
  author={Schr\"oer, Stefan},
  title={On genus change in algebraic curves over imperfect fields},
  journal={Proc. Amer. Math. Soc.},
  volume={137},
  date={2009},
  number={4},
  pages={1239-1243},
}

\bib{SiSt16}{article}{
  author={Simarra Ca{\~n}ate, Alejandro},
  author={St\"ohr, Karl-Otto},
  title={Fibrations by non-smooth projective curves of arithmetic genus two in characteristic two},
  journal={J. Pure Appl. Algebra},
  volume={220},
  date={2016},
  number={9},
  pages={3282--3299},
}

\bib{Sti78}{article}{
  author={Stichtenoth, Henning},
  title={Zur Konservativit\"at algebraischer Funktionenk\"orper},
  journal={J. Reine Angew. Math.},
  volume={301},
  date={1978},
  pages={30--45},
}

\bib{St88}{article}{
  author={St\"ohr, Karl-Otto},
  title={On singular primes in function fields},
  journal={Arch. Math. (Basel)},
  volume={50},
  date={1988},
  number={2},
  pages={156--163},
}

\bib{St04}{article}{
  author={St\"ohr, Karl-Otto},
  title={On Bertini's theorem in characteristic $p$ for families of canonical curves in $\PP ^{(p-3)/2}$},
  journal={Proc. London Math. Soc. (3)},
  volume={89},
  date={2004},
  number={2},
  pages={291--316},
}

\bib{Tate52}{article}{
  author={Tate, John},
  title={Genus change in inseparable extensions of function fields},
  journal={Proc. Amer. Math. Soc.},
  volume={3},
  date={1952},
  pages={400--406},
}

\bib{Vol91}{article}{
  author={Voloch, Jos\'e Felipe},
  title={A Diophantine problem on algebraic curves over function fields of positive characteristic},
  journal={Bull. Soc. Math. France},
  volume={119},
  date={1991},
  number={1},
  pages={121--126},
}

\bib{Zar44}{article}{
  author={Zariski, Oscar},
  title={The theorem of Bertini on the variable singular points of a linear system of varieties},
  journal={Trans. Amer. Math. Soc.},
  volume={56},
  date={1944},
  pages={130--140},
}
\end{biblist}
\end{bibdiv}
\end{document}